\documentclass[12pt]{amsart}

\usepackage{graphicx}
\newtheorem{theorem}{Theorem}[section]
\newtheorem{corollary}[theorem]{Corollary}
\newtheorem{remark}[theorem]{Remark}
\newtheorem{proposition}[theorem]{Proposition}
\newtheorem{lemma}[theorem]{Lemma}

\newtheorem{definition}[theorem]{Definition}
\newtheorem{example}[theorem]{Example}

\begin{document}

\title[On Yamada polynomial of spatial graphs]{On Yamada polynomial of spatial graphs obtained by edge replacements}

\author{Miaowang LI}
\address{School of Mathematical Sciences, Dalian University of Technology, Dalian 116024, P. R. China}
\email {limiaowang@yeah.net}

\author{Fengchun LEI}
\address{School of Mathematical Sciences, Dalian University of Technology, Dalian 116024, P. R. China}
\email{fclei@dlut.edu.cn}

\author{Fengling LI}
\address{School of Mathematical Sciences, Dalian University of Technology, Dalian 116024, P. R. China}
\email{dutlfl@163.com}

\author{Andrei VESNIN}
\address{Novosibirsk State University and Sobolev Institute of Mathematics Novosibirsk, 630090, Russia}
 \email{vesnin@math.nsc.ru}

\thanks{
The second author is supported in part by grants ( No.11329101 and No.11431009 ) of NSFC; the third author is supported in part by grants ( No.11671064 and No.11471151 ) of NSFC; the fourth author is supported by the Laboratory of Topology and Dynamics, Novosibirsk State University (a grant no. 14.Y26.31.0025 of the government of the Russian Federation) and a grant RFBR-16-01-00414.}

\subjclass[2010]{Primary 57M15; Secondary 05C31}

\keywords{Yamada polynomial, spatial graph, chain polynomial}

\begin{abstract}
We present formulae for computing the Yamada polynomial of spatial graphs obtained by replacing edges of plane graphs, such as cycle-graphs, theta-graphs, and bouquet-graphs, by spatial parts. As a corollary, it is shown that zeros of Yamada polynomials of some series of spatial graphs are dense in a certain region in the complex plane, described by a system of inequalities. Also, the relation between Yamada polynomial of graphs and the chain polynomial of edge-labelled graphs is obtained.
\end{abstract}

\maketitle

\section{Introduction}

Spatial graph theory on intrinsic knotting and linking of graphs in $S^3$ developed in the 1980s since J.H.~Conway and C.~Gordon~\cite{CC} proved that any embedding of the complete graph $\mathbf K_7$ in $\mathbb R^3$ contains a knotted cycle and any embedding of the complete graph $\mathbf K_{6}$ in $\mathbb R^{3}$ contains a pair of linked cycles. Being motivated by problems on knotting and linking of DNA and chemical compounds, the study of spatial graphs is in the center of interest for last decades. Thus, J.~Simon~\cite{S} discussed the chirality of an embedding of the complete graph $\mathbf K_5$, which answered a question raised by D.M.~Walba~\cite{W}. C.~Ernst and D.W.~Sumners~\cite{ES} introduced tangle theory of site specific recombination and gave an approach to study the behavior of DNA. X.-S.~Cheng, Y.~Lei and W.~Yang~\cite{CLY} considered invariants of double crossover links characterize topological properties of double crossover DNA polyhedra.

The modern study of spatial graphs and their generalizations combines topological and graph-theoretical methods. The powerful of polynomial invariants of knot as well as polynomial invariants for graphs was a natural motivation for investigation of polynomial invariants of spatial graphs initiated by L.H.~ Kauffman~\cite{k}. There are a number of invariants which connect spatial graphs and knots. Y.~Ohyama and K.~Taniyama~\cite{OT} explored relations among the Vassiliev invariants of knots contained in certain graphs. In 2015, N. Chbili~\cite{NC} gave criteria for a spatial graph to be (p,q)-lens graph. In 2017, A.~Henrich and L.H.~Kauffman~\cite{HK} provided a topological invariant for pseudoknots and four-valent rigid vertex spatial graphs.

It is well known~\cite{K1,K2} that the Alexander ideal and Alexander polynomial are invariants of spatial graphs which are determined by the fundamental groups of the complements of spatial graphs. In 1989, S.~Yamada~\cite{YA} introduced Yamada polynomial of spatial graphs in $R^3$, which can distinguish some non-isotopy spatial graphs with the same fundamental group. The Yamada polynomial is an concise and useful ambient isotopy invariant for graphs with maximal degree less than four. There are many interesting results on Yamada polynomial and its generalizations. J. Murakami~\cite{M} investigated the two-variable extension $\mathbf{Z}_S$ of the Yamada polynomial and gave an invariant related to the HOMFLY polynomial. In 1994, the crossing number of spatial graphs in terms of the reduced degree of Yamada polynomial has been studied by T.~Motohashi, Y.~Ohyama and K.~Taniyama~\cite{MO}. In 1996, A.~Dobrynin and A.~Vesnin~\cite{DV} studied properties of the Yamada polynomial of spatial graphs. For any graph $G$, V.~Vershinin and A.~Vesnin~\cite{VV} defined bigraded cohomology groups whose graded Euler characteristic is a multiple of the Yamada polynomial of $G$. A polynomial invariant of  virtual graphs was constructed by Y.~Miyazawa~\cite{MY} as an extension of the Yamada polynomial in 2006. Another invariant of spatial graphs associated with $U_{q}(sl(2,C))$ was introduced by S.~Yoshinaga~\cite{YO}. See~\cite{Kobe} about the relation between Yamada polynomial and Yoshinaga polynomial. Nice results on the structure of the Yamada and  flows polynomials of cubic graphs are established by I.~Agol and V.~Krushkal~\cite{AK}.

Zeros of polynomial invariants of knots and graphs is a question of special interest studied by A.D.~Sokal~\cite{SO} and  P.~Csikv\'ari, P.E.~Frenkel, J.~Hladk\'y, T.~Hubai~\cite{CFHH}  for chromatic polynomial; by O.T.~Dashbach, T.D.~Le, X.-S.~Lin~\cite{DLL} and  X.~Jun, F.~Zhang, F.~Dong, E.G.~Tay~\cite{XF} for Jones polynomial.

This paper consists of three parts. In the first part (Sections~2 and~3) we recall properties of Yamada polynomial of graphs and obtain some formulae for computing the Yamada polynomial of graphs by edge replacements via the chain polynomial (see Theorem~\ref{tger}). In the second part (Sections~4 and~5) we give formulae for computing the Yamada polynomial of spatial graphs obtained by replacing edges of  cycle graphs, theta-graphs, or bouquet graphs by spatial parts  (see Theorem~\ref{tcn}). In the last part, Section~6, we prove that zeros of the Yamada polynomial of spatial graphs are dense in a certain region in the complex plane, described by a system of inequalities (see Theorem~\ref{tmain3}).

\section{Yamada polynomial of a graph}

We consider a graph $G$, admitting loops and multiple edges. Let us use standard notations $p(G)$ and $q(G)$ for number of vertices and number of edges of it.

Before defining the Yamada polynomial of a graph, we recall a graph invariant which is a special case of the Negami polynomial invariant~\cite{N}.

\smallskip

\begin{definition}~\cite{YA}
Define 2-variable Laurent polynomial $h(G) = h(G) (x,y)$ of graph $G=(V,E)$, where $V = V(G)$ is the vertex set and $E = E(G)$ is the edge set of $G$, by the rule
$$
h(G)(x,y) = \sum_{F\subset E}(-x)^{-|F|} f(G-F),
$$
with $f(G)=x^{\mu(G)}y^{\beta(G)}$, where $\mu(G)$ and $\beta(G)$ is the number of connected components of $G$ and the first Betti number of $G$.
\end{definition}

\smallskip

\begin{definition}~\cite{YA}
{\it The Yamada polynomial} of a graph $G$ is a Laurent polynomial in one variable, obtained by the following substitution into $h(G)(x,y)$:
$$
H(G)(A) = h(G)(-1,-A-2-A^{-1}).
$$
\end{definition}

\smallskip

The following properties of $H(x,y)$ hold (see~\cite{YA} for details):
\begin{itemize}
\item[(1)] $H(\cdot)= -1$.
\item[(2)] Let $e$ be a non-loop edge of a graph $G$. Then $H(G)=H(G/e)+H(G-e)$, where $G/e$ is the graph obtained by contracting the edge $e$, and $G-e$ is the graph obtained by deleting of the edge $e$.
\item[(3)] Let $e$ be a loop edge of a graph $G$. Then $H(G)=-\sigma H(G-e)$, where $\sigma= A+1+A^{-1}$.
\item[(4)] Let $G_{1}\cup G_{2}$ be a disjoint union of graphs $G_{1}$and $G_{2}$, then $H(G_{1}\cup G_{2})= H(G_{1})H(G_{2})$.
\item[(5)] Let $G_{1}\cdot G_{2}$ be a union of graphs $G_{1}$ and $G_{2}$ having one common point, then
 \[H(G_{1}\cdot G_{2})= -H(G_{1})H(G_{2}).\]
\item[(6)] If $G$ has an isthmus, then  $H(G)= 0$.
\end{itemize}

It is easy to find directly (see also~\cite{DV}) polynomial $H(G)$ for some simple classes of graphs.

\smallskip

\begin{lemma}  \label{lemma-dv} The following properties holds with $\sigma = A + 1 + A^{-1}$.
\begin{itemize}
\item[(i)] Let $T_q$ be a tree with $q$ edges. Then $H(T_q)=0$ for all $q$.
\item[(ii)] Let $C_{n}$ be the cycle of length $n$. Then $H(C_{n})= \sigma$ for all $n$.
\item[(iii)] Let $B_q$ be the one-vertex graph with $q$ loops, also known as ``$q$-bouquet''. Then $H(B_q) \, = \,(-1)^{q-1} \sigma^q$.
\item[(iv)] Let $\Theta_s$ be the graph consisting of two vertices and $s$ edges between them, also known as ``$s$-theta-graph'' (see Fig.~\ref{fig1}). Then
$$
 H(\Theta_s) \, = \, \frac{1}{\sigma+1} [\sigma +(-\sigma)^s].
$$
\end{itemize}
\end{lemma}

\begin{figure}[!htb]
\centering
\scalebox{0.8}{\includegraphics{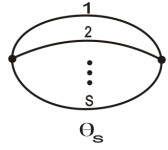}}
\caption{The graph $\Theta_{s}$.} \label{fig1}
\end{figure}

\smallskip

\begin{proposition}\label{ptv}
Let $G_{1}:G_{2}$ be the union of two graphs $G_{1}$ and $G_{2}$ having only two common vertices $u$ and $v$. Let $K_{1}$ and $K_{2}$ be graphs obtained from $G_{1}$ and $G_{2}$, respectively, by identifying $u$ and $v$. Then
\begin{eqnarray}\label{twovert}
H(G_{1}:G_{2}) & = & \frac{1}{\sigma} \Big[ H(K_{1}) H(K_{2}) + (\sigma+1) H(G_{1}) H(G_{2})  \cr
& & \qquad + H(K_{1}) H(G_{2}) + H(K_{2}) H(G_{1}) \Big].
\end{eqnarray}
\end{proposition}

\smallskip

\begin{proof}
Denote $G = G_{1}:G_{2}$ and prove the statement using induction by $q(G)$, the number of edges of the graph $G$.

If $q(G)=0$,  then $G_{1}$ and $G_{2}$ consist of isolated vertices. Suppose that $p(G_{1}) =p_{1}\geq 2$ and $p (G_{2})  = p_{2}\geq 2$. Then $G$ consists of $p_{1}+p_{2}-2$ isolated vertices, that is a disjoint union of $p_{1}+p_{2}-2$ one-vertex graphs. By properties (1) and (4), $H(G) = (-1)^{p_{1}+p_{2}-2}$. Since $p(K_{1})=p_{1}-1$ and $p(K_{2}) = p_{2}-1$, for the right-hand side of Eq.(\ref{twovert}) we get:
$$
\frac{1}{\sigma} \Big[ (-1)^{p_{1} + p_{2}-2} + (\sigma + 1) (-1)^{p_{1} + p_{2}} + (-1)^{p_{1} -1 + p_{2}} + (-1)^{p_{2}-1+p_{1}} \Big] = (-1)^{p_{1}+p_{2}}.
$$
Hence, the equality Eq.(\ref{twovert}) holds.

Now suppose that the statement holds for all graphs which $q(G) < k$, where $k\geq 1$. Consider $G=G_{1}:G_{2}$ such that $q(G) = k$.  Let $e$ be an edge of $G$. Without loss of generality, we can suppose $e \in E(G_{1})$. There are three cases to be considered.

\underline{Case~1.} Let $e$ be a loop of $G_{1}$. Then $e$ is also a loop of $K_{1}$. Hence, we have
\begin{eqnarray*}
H(G_{1}:G_{2})  & = &  -\sigma H(G_{1} - e : G_{2})  \\
& = &  - \Big[ H(K_{1} - e) H(K_{2}) \ + (\sigma+1) H(G_{1} - e) H(G_{2}) \\
& & \quad  + H(K_{1} -e) H(G_{2}) + H(K_{2})H(G_{1}-e) \Big]  \\
&=& \frac{1}{\sigma} \Big[ H(K_{1}) H(K_{2}) + (\sigma+1)H(G_{1}) H(G_{2}) \\
& &   \quad + H(K_{1}) H(G_{2}) + H(K_{2}) H(G_{1}) \Big].
\end{eqnarray*}

\underline{Case~2.} Let $e$ be neither a loop nor an edge with two end vertices $u$ and $v$, where $u$ and $v$ are common vertices of $G_1$ and $G_2$. Then $e$ is a non-loop edge in $K_{1}$. Hence,
\begin{eqnarray*}
H(G/e)  & =  & H(G_{1}/e : G_{2}) \\
&  = &   \frac{1}{\sigma} \Big[  H(K_{1}/e )H(K_{2}) + (\sigma+1) H((G_{1}/e) H(G_{2})  \\
& & \qquad   + H(K_{1}/e) H(K_{2}) + H(K_{2}) H(G_{1}/e) \Big]
\end{eqnarray*}
and
\begin{eqnarray*}
H(G-e) &= & H(G_{1}-e:G_{2}) \\
& = &\frac{1}{\sigma} \Big[ H(K_{1}-e) H(K_{2}) + (\sigma+1) H(G_{1}-e) H(G_{2})  \\
& &  \qquad  + H(K_{1}-e) H(G_{2}) + H(K_{2}) H(G_{1}-e) \Big].
\end{eqnarray*}
Therefore,
\begin{eqnarray*}
H(G)&= & H(G/e) + H(G-e) \\
&= &\frac{1}{\sigma} \Big[ H(K_{1}/e) H(K_{2}) + (\sigma+1) H(G_{1}/e) H(G_{2}) \\
& &\quad  + H(K_{1}/e)H(G_{2}) + H(K_{2})H(G_{1}/e) + H(K_{1}-e) H(K_{2}) \\
& & \quad + (\sigma+1) H(G_{1}-e) H(G_{2})+ H(K_{1}-e) H(G_{2}) \\
& & \quad +H(K_{2}) H(G_{1}-e) \Big] .
\end{eqnarray*}
Hence
\begin{eqnarray*}
H(G) &=& \frac{1}{\sigma} \Big[ \Big( H(K_{1}/e) + H(K_{1}-e) \Big) H(K_{2}) \\
& & \quad + (\sigma+1) \Big( H((G_{1}/e) + H(G_{1}-e) \Big) H(G_{2}) \\
& & \quad + \Big( H(K_{1}/e) + H(K_{1}-e) \Big) H(G_{2}) \\
& & \quad  + H(K_{2})  \Big( H(G_{1}/e)+H(G_{1}-e) \Big) \Big]  \\
 &=& \frac{1}{\sigma} \Big[ H(K_{1})H(K_{2})+(\sigma+1)H(G_{1})H(G_{2})+H(K_{1})H(G_{2}) \\
 & & \quad +H(K_{2})H(G_{1}) \Big] .
\end{eqnarray*}

\underline{Case 3.} Let $e$ be an edge with two end vertices $u$ and $v$, where $u$ and $v$ are common vertices of $G_1$ and $G_2$. Then $e$ is a loop in $K_{1}$. Hence,  by properties (5) and (3) we have
$$
H(G/e) = H(G_{1}/e:G_{2}) = -H(K_{1}-e)H(K_{2}) = \frac{1}{\sigma}H(K_{1})H(K_{2})
$$
and
\begin{eqnarray*}
H(G-e) & =& H(G_{1}-e:G_{2}) \\
& = &\frac{1}{\sigma} \Big[ H(K_{1}-e) H(K_{2}) + (\sigma+1) H(G_{1}-e) H(G_{2}) \\
& & \qquad   + H(K_{1}-e) H(G_{2}) + H(K_{2}) H(G_{1}-e) \Big] \\
&=&\frac{1}{\sigma} \Big[ -\frac{1}{\sigma} H(K_{1}) H(K_{2}) + (\sigma+1) \Big( H(G_{1})+\frac{1}{\sigma}H(K_{1}) \Big) H(G_{2})  \\
& & \qquad -\frac{1}{\sigma} H(K_{1}) H(G_{2}) + H(K_{2}) \Big( H(G_{1})  + \frac{1}{\sigma} H(K_{1}) \Big) \Big] .
\end{eqnarray*}
Therefore,
\begin{eqnarray*}
H(G) &= & H(G/e) + H(G-e)  =  \frac{1}{\sigma} H(K_{1}) H(K_{2}) \\  & & +  \frac{1}{\sigma} \Big[ - \frac{1}{\sigma} H(K_{1}) H(K_{2})  + (\sigma+1) H(G_{2}) \Big( H(G_{1}) + \frac{1}{\sigma} H(K_{1}) \Big) \\
& &  \qquad  -\frac{1}{\sigma} H(K_{1}) H(G_{2}) + H(K_{2}) \Big( H(G_{1}) + \frac{1}{\sigma} H(K_{1}) \Big) \Big] \\
&=& \frac{1}{\sigma} \Big[ H(K_{1}) H(K_{2}) + (\sigma+1) H(G_{1}) H(G_{2}) + H(K_{1}) H(G_{2})\\
& & \quad + H(K_{2}) H(G_{1}) \Big] .
\end{eqnarray*}
This completes the proof.
\end{proof}

\section{Yamada polynomials of a graph obtained by edge replacements}

We recall some properties of the chain polynomial introduced by R.C.~Read and E.G.~Whitehead~\cite{RE}, see also~\cite{XF}.  The chain polynomial is  defined on edge-labelled graphs with labels are elements of a commutative ring with unity. We will denote the edges by the labels associated with them.

\begin{definition}
{\it The chain polynomial} $\operatorname{Ch}(G)$ of a labelled graph $G$ is defined as
$$
\operatorname{Ch}(G) = \sum_{Y \subset E} F_{G-Y} (1-w) \prod_{a \in Y} a,
$$
where the sum is taken over all subsets of the edge set $E$ of $G$, $F_{G-Y} (1-w)$ denotes the flow polynomial of the subgraph $G - Y$ calculated at $1-w$, and $\prod_{a\in Y}$ denotes the product of edge-labels of $Y$.
\end{definition}

The chain polynomial can be also defined in the following recursive form.

\begin{definition} \label{def3.2}
{\it The chain polynomial}  $\operatorname{Ch} (G) (w)$ in a variable $w$ of a labelled graph $G$ is defined by following rules.
\begin{itemize}
\item[(1)] If $G$ is edgeless, then $\operatorname{Ch}(G)=1$.
\item[(2)] Otherwise, suppose $a$ is an edge of $G$ labelled by $a$. Then
\begin{itemize}
\item[(2a)] If $a$ is a loop, then $\operatorname{Ch}(G) = (a-w) \operatorname{Ch}(G-a)$.
\item[(2b)] If $a$ is not a loop, then $\operatorname{Ch}(G) = (a-1) \operatorname{Ch} (G-a) + \operatorname{Ch}(G/a)$.
\end{itemize}
\end{itemize}
\end{definition}

\begin{example}\label{Ecn}
Let $C_{n}$ be the $n$-cycle with edges labelled by $a_{1}, a_{2}, \cdots, a_{n}$, then
$$
\operatorname{Ch}(C_{n}) = \prod_{i=1}^n a_i \, - \, w.
$$
\end{example}

\begin{example}
Let $\Theta_{s}$ be the $s$-theta-graph with edges labelled by \linebreak $a_{1}, a_{2}, \cdots, a_{s}$, then
$$
\operatorname{Ch} (\Theta_{s}) = \frac{1}{1-w} \Big[ \prod^{s}_{i=1}(a_{i}-w) - w \prod ^s_{i=1}(a_{i}-1) \Big].
$$
\end{example}

\begin{example}
Let $B_{q}$ be the $q$-bouquet with $q$ loops labelled by \linebreak  $a_{1}, a_{2}, \ldots,  a_{q}$, then
$$
\operatorname{Ch} ( B_{q}) = \prod^{И}_{i=1}(a_{i}-w).
$$
\end{example}

To explore the relation between the chain polynomial and the Yamada polynomial, inspired by \cite{XF}, let us give the following notation. Let $G$ be a connected labelled graph. Denote by $\widetilde{G}$ the graph obtained from $G$ by replacing each edge $a= uv$ of $G$ by a connected graph $K_{a}$ with two attached vertices $u$ and $v$ that has only the vertices $u$ and $v$ in common with $\widetilde{G-a}$.

Let $K'_{a}$ be the graph obtained from $K_{a}$ by identifying $u$ and $v$, the two attached
vertices of $K_{a}$. Denote
$$
\alpha_{a }= \alpha(K_{a}) := \frac{1}{\sigma}[(\sigma+1)H(K_{a})+H(K'_{a})],
$$
$$
\beta_{a} = \beta(K_{a}) := \frac{1}{\sigma}[H(K_{a})+H(K'_{a})],
$$
and
$$
\gamma_{a} = \gamma(K_{a}) : =1-\frac{\alpha (K_{a})}{\beta (K_{a})}.
$$
It is easy to see that
$$
H(K'_{a})=(\sigma+1)\beta_{a}-\alpha_{a} \qquad \text{and} \qquad  H(K_{a})=\alpha_{a}-\beta_{a}.
$$

The following result gives the relation between Yamada polynomial and chain polynomial.

\begin{theorem}\label{tger}
Let $G$ be a connected labelled graph, and $\widetilde{G}$ be the graph obtained from $G$ by replacing the edge $a$ by a connected graph $K_{a}$ for every edge $a$ of $G$. If we replace $w$ by $-\sigma$, and replace $a$ by $\gamma_{a}$ for every label $a$ in $\operatorname{Ch}(G)$, then we get
\begin{equation}
H(\widetilde{G})=\frac{\prod_{a\in E(G)}\beta_{a}}{(-1)^{q(G)-p(G)}} \operatorname{Ch}(G),
\end{equation}
where $p(G)$ and $q(G)$ are the numbers of vertices and edges of $G$, respectively.
\end{theorem}

\begin{proof}
Denote
$$
T(G) = \frac{(-1)^{q(G)-p(G)} H(\widetilde{G})}{\prod_{a\in E(G)}\beta_{a}}.
$$
To prove the statement we will show that $T(G)$ satisfies conditions from Definition~3.2 of $\operatorname{Ch}(G)$ if we identify $w$ with $-\sigma$ and $a$ with $\gamma_{a}$. Let us check these conditions.

(1) If $G$ is edgeless, then $E(G) = \emptyset$, $E(\widetilde{G}) = \emptyset$ and  $H(\widetilde{G}) =  (-1)^{p(\widetilde{G})} = (-1)^{p(G)}$. Then
$$
T(G)=\frac{(-1)^{q(G)-p(G)} H(\widetilde{G})  }{\prod_{a\in E(G)}\beta_{a}}= (-1)^{-p(G)} (-1)^{p(G)} = 1.
$$

(2a) If $a$ is a loop of $G$, then
\begin{eqnarray*}
T(G) & = & \frac{(-1)^{q(G)-p(G)} H(\widetilde{G})}{\prod_{a\in E(G)}\beta_{a}} \\
& = & \frac{(-1)^{q(G-a)+1-p(G-a)} \Big[-H(K_{a'}) H(\widetilde{G-a})\Big]}{\beta_{a}\prod_{a\in E(G-a)}\beta_{a}} \\
& = & \frac{(-1)^{q(G-a)+1-p(G-a)} \Big[ \Big( \alpha_{a} - (\sigma+1) \beta_{a} \Big) H(\widetilde{G-a}) \Big]}{\beta_{a} \prod_{a\in E(G-a)}\beta_{a}} \\
&=& \frac{- \Big( \alpha_{a}-(\sigma+1)\beta_{a}) \Big)}{\beta_{a}}T(G-a) =(\gamma_{a}+\sigma)T(G-a).
\end{eqnarray*}
Compare with $\operatorname{Ch}(G) = (a - w) \operatorname{Ch} (G-a)$ from Definition~3.2.

(2b) If $a \in E(G)$ is not a loop, then by Proposition~\ref{ptv}
\begin{eqnarray*}
H(\widetilde{G}) &=& H(\widetilde{G-a}:K_{a}) \\
& = & \frac{1}{\sigma} \Big[ H(\widetilde{G/a}) H(K'_{a} )+ (\sigma+1) H(\widetilde{G-a}) H(K_{a}) \\
& & \qquad \qquad \qquad \quad  + H(K'_{a}) H(\widetilde{G-a}) + H(K_{a}) H(\widetilde{G/a})\Big] \\
&= & \frac{1}{\sigma} \Big[ H(\widetilde{G/a}) \Big( H(K'_{a})+H(K_{a}) \Big) \\
&&   \qquad \qquad \qquad \quad+ H(\widetilde{G-a})\Big( (\sigma+1)H(K_{a})+H(K'_{a})  \Big) \Big] \\
&=& \beta_{a}H(\widetilde{G/a}) + \alpha_{a}H(\widetilde{G-a}).
\end{eqnarray*}
Therefore,
\begin{eqnarray*}
T(G) & =  & \frac{(-1)^{q(G)-p(G)}H(\widetilde{G})}{\prod_{a\in E(G)}\beta_{a}} \\
& = & \frac{(-1)^{q(G-a)+1-p(G-a)} \Big[ \beta_{a} H(\widetilde{G/a}) + \alpha_{a} H(\widetilde{G-a}) \Big]}{\beta_{a}\prod_{a\in E(G-a)}\beta_{a}}\\
&=&\frac{(-1)^{q(G/a)+1-p(G/a)-1}  \beta_{a} H(\widetilde{G/a}) }{\beta_{a}\prod_{a\in E(G/a)}\beta_{a}} \\
& & \quad +
\frac{(-1)^{q(G-a)+1-p(G-a)} \alpha_{a}H(\widetilde{G-a}) }{\beta_{a}\prod_{a\in E(G-a)}\beta_{a}}\\
&=&T(G/a) + \Big(-\frac{\alpha_{a}}{\beta_{a}} \Big) T(G-a) .
\end{eqnarray*}
Thus,
$$
T(G) = T(G/a) +(\gamma_{a}-1) T(G-a).
$$
Compare with $\operatorname{Ch} (G) = \operatorname{Ch} (G/a) +  (a-1) \operatorname{Ch} (G-a)$ from Definition~3.2.
\end{proof}

\section{Yamada polynomial of spatial graphs}

Next we will consider the Yamada polynomial of spatial graphs. Through the paper we work in the piecewise-linear category. Let $G$ be a graph embedded in $\mathbb R^3$, and $g$ be a diagram of  $G$. For any double point, S.~Yamada~\cite{YA} defined the spin of $+1$, $-1$ and $0$, which are denoted by $S_+$, $S_-$ and $S_0$, as shown in Fig.~\ref{fig2}.

\begin{figure}[!htb]
\centering
\scalebox{0.8}{\includegraphics{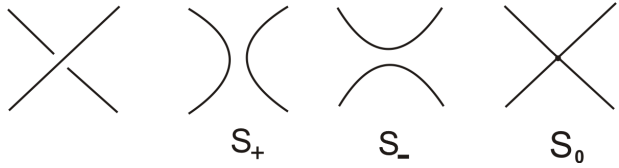}}
\caption{Spin of $+1$, $-1$ and $0$.} \label{fig2}
\end{figure}

Let $S$ be the plane graph obtained from $g$ by replacing each double point with a spin.  $S$ is called {\it a state} on $g$. Denote the set of all states by $U(g)$. Put
$$
c(g|S)= A^{m_1-m_2},
$$
where $m_1$ and $m_2$ are the numbers of double points with spin $S_+$ and $S_-$, respectively, used to obtain $S$ from $g$.

\smallskip

\begin{definition}~\cite{YA}\label{dsg}
{\it The Yamada polynomial} of a diagram $g$ of a spatial graph $G$ is a Laurent polynomial in variable $A$ defined as follows
$$
R[g] = R[g](A) : =\sum_{S\in U(g)} c(g|S)H(S),
$$
where $H(S) = h(S)(-1,-A-2-A^{-1})$.
\end{definition}

For interesting properties of the Yamada polynomial and its invariantness under transformations of spatial graph diagrams see~\cite{YA}. Here we recall only some of them.

\smallskip

\begin{proposition}~\cite{YA}\label{pr}
The following properties hold.
\begin{itemize}
\item[(1) ]Let $g_{1}\cup g_{2}$ be a disjoint union of diagrams $g_{1}$ and $g_{2}$, then
$$R[g_{1}\cup g_{2}]= R[g_{1}]R[g_{2}].$$
\item[(2)]  Let $g_{1}\cdot g_{2}$ be a union of diagrams $g_{1}$ and $g_{2}$ having one common point, then
$$
R[g_{1}\cdot g_{2}] = -R[g_{1}]R[g_{2}].
$$
\item[(3)] If $g$ has an isthmus, then  $R[g]= 0$.
\end{itemize}
\end{proposition}

\smallskip
\begin{remark}
If a diagram $g$ of $G$ does not have double points, then $R[g]= H(G)$.
\end{remark}

Proposition \ref{ptv} about $H(G)$ implies the similar  formula for polynomial $R[g]$.

\begin{proposition}\label{cdtv}
Let  $g_{1}:g_{2}$ be the union of two diagrams $g_{1}$ and $g_{2}$ having only two common vertices $u$ and $v$. Let $k_{1}$ and $k_{2}$ be diagrams obtained from $g_{1}$ and $g_{2}$, respectively, by identifying $u$ and $v$. Then
\begin{equation}
\begin{gathered}
R[g_{1}:g_{2}] = \frac{1}{\sigma} \Big( R[k_{1}] R[k_{2}] + (\sigma+1) R[g_{1}] R[g_{2}] \\ + R[k_{1}] R[g_{2}] + R[k_{2}] R[g_{1}] \Big).
\end{gathered}
\end{equation}
\end{proposition}

\begin{proof}
By Definition \ref{dsg} and Proposition \ref{ptv}, we obtain
\begin{eqnarray*}
R[g_{1}:g_{2}] & = & \sum_{S\in U(g_{1}:g_{2})} c(g_{1}:g_{2}|S) H(S) \\
&=& \sum_{S_{1}\in U(g_{1})} \sum_{S_{2}\in U(g_{2})} c(g_{1}|S_1) c(g_{2}|S_2) H(S_{1}:S_2) \\
&=&\sum_{S_{1}\in U(g_{1})} \sum_{S_{2}\in U(g_{2})} c(g_{1}|S_1) c(g_{2}|S_2) \frac{1}{\sigma} \Big[ H(S_{1}')H(S_{2}') \\
& & \quad + (\sigma+1) H(S_{1}) H(S_{2}) +H(S_{1})H(S_{2}')+H(S_{1}')H(S_{2}) \Big]\\
&=&\frac{1}{\sigma}  \Big(  \sum_{S_{1} \in U(g_{1})} \Big[ c(g_{1}|S_1) H(S_{1}') \Big] \sum_{S_{2}\in U(g_{2})} \Big[ c(g_{2}|S_2) H(S_{2}') \Big]  \\
& &  + (\sigma+1) \sum_{S_{1} \in U(g_{1})} \Big[ c(g_{1}|S_1) H(S_{1}) \Big] \sum_{S_{2}\in U(g_{2})} \Big[ c(g_{2}|S_2) H(S_{2}) \Big] \\
&&  +  \sum_{S_{1}\in U(g_{1})} \Big[ c(g_{1}|S_1)H(S_{1}') \Big]  \sum_{S_{2}\in U(g_{2})} \Big[ c(g_{2}|S_2) H(S_{2}) \Big] \\
& & +\sum_{S_{1}\in U(g_{1})}  \Big[ c(g_{1}|S_1) H(S_{1}) \Big]  \sum_{S_{2} \in U(g_{2})}  \Big[c(g_{2}|S_2) H(S_{2}') \Big] \Big) \\
&=&\frac{1}{\sigma} \Big[ R[k_{1}]R[k_{2}] + (\sigma+1) R[g_{1}] R[g_{2}] + R[k_{1}] R[g_{2}] \\ & & \qquad \qquad + R[k_{2}] R[g_{1}] \Big].
\end{eqnarray*}
This completes the proof.
\end{proof}

\section{Yamada Polynomials of diagrams obtained by edge replacements}

Let $G$ be a connected plane labelled graph, where each edge $a$ with terminal vertices $u$ and $v$ is labelled by a diagram $g_{a}$ having  two vertices $u_{a*}$ and $v_{a*}$ indicated. We define $G(g_{a})$ to be the spatial graph (its diagram) obtained from $G$ by replacing an edge $a= u_{a}v_{a}$ of $G$ by a connected diagram $g_{a}$ with identification $u_{a}$ with $u_{a*}$ and $v_{a}$ with $v_{a*}$ in such a way that  $g_{a}$ has only vertices $u_{a}=u_{a*}$ and $v_{a}=v_{a*}$ in common with $G-a$. We define $G(g_{a}, g_{b})$ to be the spatial graph (its diagram) obtained from $G(g_{a})$ by replacing an edge $b= u_{b}v_{b}$ of $G(g_{a})\cap (G-a)$  by a connected diagram $g_{b}$ with identification $u_{b}$ with $u_{b*}$ and $v_{b}$ with $v_{b*}$ in such a way that  $g_{b}$ has only vertices $u_{b}=u_{b*}$ and $v_{b}=v_{b*}$ in common with $G(g_{a})$. With the same construction, we can obtain $G(g_{a_1}, g_{a_2}, g_{a_3}, \ldots)$.
In this context we denote by $g_{a}'$ the diagram obtained from $g_{a}$ by identifying vertices $u_{a*}$ and $v_{a*}$ in such a way that no new intersections appear.

\begin{theorem}\label{tcn} The following properties hold.
\begin{itemize}
\item[(1)] Let $C_{n}$ be the $n$-cycle graph with edges labelled by $a_{1}, a_{2}, \ldots, a_{n}$, then
\begin{equation}\label{ecn}
R[C_{n}(g_{a_1}, g_{a_2}, \cdots, g_{a_n})] = \prod_{i=1}^n(-R[g_{a_i}])+\sigma\prod_{i=1}^n \Big( \frac{R[g_{a_i}]+R[g_{a_i}']}{\sigma} \Big).
\end{equation}
\item[(2)] Let $\Theta_{s}$ be the $s$-theta-graph with edges labelled by $a_{1}, a_{2}, \ldots, a_{s}$, then
\begin{equation}\label{eths}
\begin{gathered}
R[\Theta_{s}(g_{a_1}, g_{a_2}, \cdots, g_{a_s})] \quad \qquad \qquad \qquad  \qquad \qquad \qquad \qquad \qquad  \\
= \frac{(-1)^{s}}{1+\sigma} \prod_{i=1}^sR[g_{a_i}'] + \frac{\sigma}{1+\sigma} \prod_{i=1}^s \Big(\frac{\sigma+1}{\sigma}R[g_{a_i}]+\frac{1}{\sigma}R[g_{a_i}'] \Big).
\end{gathered}
\end{equation}
\item[(3)] Let $B_{q}$ be the $q$-bouquet with loops labelled by $a_{1}, a_{2}, \ldots, a_{q}$, then
\begin{equation}\label{ebp}
R[B_{q}(g_{a_1}, g_{a_2}, \cdots, g_{a_q})] = (-1)^{q-1}\prod_{i=1}^q R[g_{a_i}'].
\end{equation}
\end{itemize}
\end{theorem}

\begin{proof}
(1) Firstly, we prove Eq.(\ref{ecn}) by induction on $n$. Recall, that by Example~\ref{Ecn}, the chain polynomial of the $n$-cycle graph is $$\operatorname{Ch} (C_n) = \prod_{i=1}^n a_i-w.$$

If $n=1$, then $R[C_1(g_{a_1})]=R[g_{a_1}']$, so Eq.(\ref{ecn}) holds.

Suppose  Eq.(\ref{ecn}) holds for all $n \leq k-1$, where $k \geq 2$. Observe that spatial graph  $C_{k}(g_{a_1}, g_{a_2}, \ldots, g_{a_k})$ can be presented as the union $g_{a_k} : L_{k-1}(g_{a_1}, g_{a_2}, \ldots, g_{a_{k-1}})$, where $L_{k-1}$ is a graph with $k$ vertices connected by $k-1$ edges one by one, as shown in~Fig.~\ref{fig3}.

\begin{figure}[!htbp]
\centering
{\includegraphics[scale=0.8]{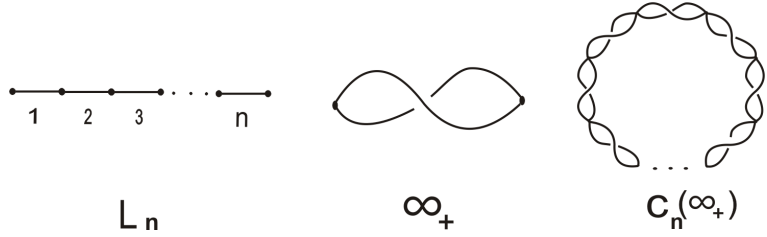}}
\caption{Graph $L_n$, diagrams $\infty_{+}$ and $C_{n}(\infty_{+})$} \label{fig3}
\end{figure}

Then
$$
R[C_{k}(g_{a_1}, g_{a_2}, \cdots, g_{a_k})]=R[g_{a_k} : L_{k-1}(g_{a_1}, g_{a_2}, \cdots, g_{a_{k-1}})].
$$
By Proposition \ref{cdtv}, we have
\begin{eqnarray*}
& & R[g_{a_k} : L_{k-1}(g_{a_1}, g_{a_2}, \cdots, g_{a_{k-1}})]\\
& & \qquad \qquad \qquad = \frac{1}{\sigma} \Big[ R[L_{k-1}(g_{a_1}, g_{a_2}, \cdots, g_{a_{k-1}})']R[g_{a_k}'] \\
& & \qquad \qquad \qquad \qquad + (1+\sigma)R[L_{k-1}(g_{a_1}, g_{a_2}, \cdots, g_{a_{k-1}})]R[g_{a_k}]\\
& & \qquad \qquad \qquad \qquad   + R[L_{k-1}(g_{a_1}, g_{a_2}, \cdots, g_{a_{k-1}})']R[g_{a_k}] \\
& & \qquad \qquad \qquad \qquad +R[L_{k-1}(g_{a_1}, g_{a_2}, \cdots, g_{a_{k-1}})]R[g_{a_k}'] \Big].
\end{eqnarray*}

Applying Proposition \ref{pr}, we get
$$
R[L_{k-1}(g_{a_1}, g_{a_2}, \cdots, g_{a_{k-1}})] = (-1)^{k}\prod_{i=1}^{k-1}R[g_{a_i}]
$$
and
$$
\begin{gathered}
R[L_{k-1}(g_{a_1}, g_{a_2}, \cdots, g_{a_{k-1}})'] = R[C_{k}(g_{a_1}, g_{a_2}, \cdots, g_{a_{k-1}})] \\ =\prod_{i=1}^{k-1} \Big( -R[g_{a_i}] \Big) + \sigma \prod_{i=1}^{k-1} \Big( \frac{R[g_{a_i}] + R[g_{a_i}']}{\sigma} \Big).
\end{gathered}
$$
Therefore,
\begin{eqnarray*}
&& R[C_{k}(g_{a_1}, g_{a_2}, \cdots, g_{a_k})]\\
& & = \frac{1}{\sigma} \Big[  \Big( \prod_{i=1}^{k-1}(-R[g_{a_i}]) + \sigma \prod_{i=1}^{k-1}\frac{R[g_{a_i}]+R[g_{a_i}']}{\sigma} \Big) R[g_{a_k}'] \\
& & \qquad \qquad +(1+\sigma)(-1)^{k} \Big( \prod_{i=1}^{k-1} R[g_{a_i}] \Big) R[g_{a_k}] \\
& & \qquad \qquad  + \Big( \prod_{i=1}^{k-1}(-R[g_{a_i}]) + \sigma \prod_{i=1}^{k-1}\frac{R[g_{a_i}] + R[g_{a_i}']}{\sigma} \Big) R[g_{a_k}] \\
& & \qquad \qquad +(-1)^{k} \Big( \prod_{i=1}^{k-1}R[g_{a_i}] \Big) R[g_{a_k}'] \Big]  \\
& & = \prod_{i=1}^k (-R[g_{a_i}]) + \sigma\prod_{i=1}^k  \Big( \frac{R[g_{a_i}] +R[g_{a_i}']}{\sigma} \Big).
\end{eqnarray*}

Thus, Eq.(\ref{ecn}) holds.

(3) Similarly to the above case, Eq.(\ref{ebp}) is equivalent to
$$
\begin{gathered}
R[B_{q}(g_{a_1}, g_{a_2}, \ldots, g_{a_q})] \qquad \qquad  \qquad \qquad \qquad \qquad \qquad \qquad  \\
=R[g_{a_1}] \, R[g_{a_2}] \cdots R[g_{a_{q-1}}] \, R[g_{a_q}]= (-1)^{q-1}\prod_{i=1}^qR[g_{a_i}'].
\end{gathered}
$$

(2) Now, we prove Eq.(\ref{eths}) by induction on $s$.  If $s=1$, then $R[\Theta_{1}(g_{a_1})]= R[g_{a_1}]$ and Eq.(\ref{eths}) holds.

Suppose that Eq.(\ref{eths}) holds for all $s \leq k-1$, where $k\geq 2$. Observe that the left-hand side of Eq.(\ref{eths}) can be presented as the union $R[\Theta_{s}(g_{a_1}, g_{a_2}, \cdots, g_{a_{s-1}}):  g_{a_s}]$.

Since
$$
R[\Theta_{s}(g_{a_1}, g_{a_2}, \ldots, g_{a_s})'] = R[B_{s}(g_{a_1}, g_{a_2}, \ldots, g_{a_s})],
$$
by Proposition \ref{cdtv} we have
\begin{eqnarray*}
\begin{gathered}
 R[\Theta_{s}(g_{a_1}, g_{a_2}, \ldots, g_{a_{s}})] \qquad \qquad \qquad \qquad \qquad  \qquad \qquad \\
 = \frac{1}{\sigma} \Big( R[\Theta_{s}(g_{a_1}, g_{a_2}, \ldots, g_{a_{s-1}})']R[g_{a_s}'] \\
 \qquad + (1+\sigma) R[\Theta_{s}(g_{a_1}, g_{a_2}, \ldots, g_{a_{s-1}})]R[g_{a_s}] \\
 \qquad +R[\Theta_{s}(g_{a_1}, g_{a_2}, \ldots, g_{a_{s-1}})'] R[g_{a_s}]
 + R[\Theta_{s}(g_{a_1}, g_{a_2}, \ldots, g_{a_{s-1}})]R[g_{a_s}'] \Big) \\
 = \frac{(-1)^{s}}{1+\sigma}\prod_{i=1}^sR[g_{a_i}'] + \frac{\sigma}{1+\sigma}\prod_{i=1}^s \Big( \frac{\sigma+1}{\sigma}R[g_{a_i}] + \frac{1}{\sigma}R[g_{a_i}'] \Big).
\end{gathered}
\end{eqnarray*}
The proof is competed.
\end{proof}

If every edge $a$ of $G$ is labelled by the same diagram $g_{a} = g$, we denote $G(g, \ldots, g)$ shortly by $G(g)$.

\begin{corollary}
If all edges of a graph $G$ are labelled by the same diagram $g$, then
\begin{equation}
R[C_{n}(g)]=(-R[g])^{n}+\sigma \Big( \frac{R[g]+R[g']}{\sigma} \Big)^{n},
\end{equation}
\begin{equation}
R[\Theta_{s}(g)] = \frac{(-1)^{s}}{1+\sigma}R[g']^{s} + \frac{\sigma}{1+\sigma} \Big( \frac{\sigma+1}{\sigma}R[g]+\frac{1}{\sigma}R[g'] \Big)^{s},
\end{equation}
and
\begin{equation}
R[B_{q}(g)] = (-1)^{q-1} R[g']^q.
\end{equation}
\end{corollary}

Remark that  $C_{n}(g_{a_1}, g_{a_2}, \ldots, g_{a_n})$ is the ``ring of beads'' discussed in~\cite{xj} and~\cite{RE}. By Eq.(\ref{ecn}), we get the following property.

\begin{corollary}
The Yamada polynomial $R[C_{n}(g_{a_1}, g_{a_2}, \ldots, g_{a_n})]$ of a  ``ring of beads''  graph $C_{n}(g_{a_1}, g_{a_2}, \ldots, g_{a_n})$ is
independent of the order in which the beads occur.
\end{corollary}

The following example illustrates the applying of Theorem \ref{tcn}.

\begin{example}\label{E4}
Let $\infty_{+}$ be the spatial graph diagram with two vertices and one double point signed by ``$+$'', see Fig.~2. We denote the mirror image of diagram $\infty_{+}$ by $\infty_{-}$. Direct calculations give Yamada polynomials $R[\infty_{+}]=A^{-2}\sigma$ and $R[\infty_{+}']=\sigma$, where $\sigma = A + 1 + A^{-1}$. By Theorem \ref{tcn}, the Yamada polynomial of $C_{n}(\infty_{+})$ is as follows:
$$
R[C_{n}(\infty_{+})]=(-A^{-2}\sigma)^{n}+\sigma(A^{-2}+1)^{n};
$$
the Yamada polynomial of $\Theta_{s}(\infty_{+})$ is as follows:
$$
R[\Theta_{n}(\infty_{+})]=\frac{(-\sigma)^s+\sigma\big((\sigma+1)A^{-2}+1\big)^s}{1+\sigma};
$$
and the Yamada polynomial of $B_{q}(\infty_{+})$ is as follows:
$$
R[B_{q}(\infty_{+})]=(-1)^{q-1} \sigma^q.
$$
\end{example}

\section{Zeros of Yamada polynomial}

In this section, we investigate the density of zeros of Yamada polynomial of two classes of spatial graphs. One is the ``ring of beads'' graphs $C_{n}(\Theta_{s})$ obtained from $C_{n}$ by replacing the edge $a$ by a connected plane $s$-theta graph $\Theta_{s}$ for every edge $a$ of $C_{n}$. The other is the class of spatial graphs $C_{n}(\Theta_{s}(\infty_{+}))$ and their  mirror images $C_{n}(\Theta_{s}(\infty_{-}))$.

To study zeros of families of polynomials we will base on the following  results by S.~Beraha, J.~Kahane, N.J.~Weiss \cite{BJ} and by A.D.~Sokal \cite{SO}.

\begin{theorem}\label{tbj} \cite{BJ}
If $\{f_n(x)\}$ is a family of polynomials such that
$$
f_n(x) = \alpha_1(x) \lambda_1(x)^n + \alpha_2(x) \lambda_2(x)^n +\ldots + \alpha_l(x) \lambda_l(x)^n,
$$
where the $\alpha_i(x)$ and $\lambda_i(x)$ are fixed non-zero polynomials, such that no pair $i\neq j$ has
$\lambda_i(x) \equiv \omega \lambda_j(x)$ for some complex number $\omega$ of unit modulus. Then $z$ is a limit of zeros of  $\{f_n(x)\}$ if and only if
\begin{itemize}
\item[(1)] two or more of the $\lambda_i(z)$ are of equal modulus, and strictly greater in modulus than the others; or
\item[(2)] for some j, the modulus of $\lambda_j(z)$ is strictly greater than those of the others, and $\alpha_i(z)=0$.
\end{itemize}
\end{theorem}

\begin{lemma} \cite[Lemma 1.6]{SO} \label{lso}
Let $F_{1}$, $F_{2}$, and $G$ be analytic functions on a disc $\{ z \in \mathbb C : |z|<R \}$ such that $|G(0)|\leq 1$ and G not constant. Then, for each $\epsilon >0$ there exists $s_{0}<\infty$ such that for all integers $s\geq s_{0}$ the equation
$$
|1+F_{1}(z)G(z)^s|=|1+F_{2}(z)G(z)^s|
$$
has a solution in the disc $|z|<\epsilon$.
\end{lemma}

Now, let us turn to families of Yamada polynomials.

\begin{theorem}\label{tmain}
Zeros of the Yamada polynomials $\{R[C_{n}(\Theta_{s})]\}_{n=1,s=1}^{\infty, \infty}$ are dense in the region $\{z\in \mathbb{C} \, : \,  |z+1+z^{-1}|\geq 1\}$.
\end{theorem}

\begin{proof}
By calculating the Yamada polynomial of $C_{n}(\Theta_{s})$  we obtain
$$
R[C_{n}(\Theta_{s})]=H(C_{n}(\Theta_{s})) = \Big( -H(\Theta_{s}) \Big)^n + \sigma \Big( \frac{H(\Theta_{s})+H(\Theta_{s}')}{\sigma} \Big)^n.
$$
We will apply Theorem~\ref{tbj} in the case $l=2$ and use condition (1). Thus, numbers $A$ satisfying the equation
\begin{equation}\label{e1}
|-H(\Theta_{s})| = \Big| \frac{H(\Theta_{s}) + H(\Theta_{s}')}{\sigma} \Big|,
\end{equation}
are exactly limits of zeros of Yamada polynomials  $\{ R[C_{n} (\Theta_{s})] \}_{n=1,s=1}^{\infty, \infty}$.

Observe that $\Theta_{s}'$ is the $s$-bouquet $B_{s}$. Since polynomials $H(\Theta_{s})$ and $H(B_{s})$ were calculated in Lemma~\ref{lemma-dv}, we have
$$
H(\Theta_{s})=\frac{\sigma-(-1)^{s-1}\sigma^s}{1+\sigma} \qquad \text{and} \qquad H(\Theta_{s}')=(-1)^{s-1}\sigma^s.
$$
Therefore, Eq.(\ref{e1}) is equivalent to
$$
\Big| \frac{\sigma-(-1)^{s-1}\sigma^s}{1+\sigma} \Big| =  \Big| \frac{1}{\sigma} \Big[ \frac{\sigma-(-1)^{s-1}\sigma^s}{1+\sigma}+(-1)^{s-1}\sigma^s \Big] \Big| ,
$$
whence
$$
|\sigma-(-1)^{s-1}\sigma^s| =   |1+(-1)^{s-1}\sigma^s|
$$
and then
\begin{equation}\label{e2}
|1+\sigma(-\sigma)^{-s}|=|1-(-\sigma)^{-s}|.
\end{equation}

Recall that $\sigma = A + 1 + A^{-1}$. Let $A_0$ be a complex number with $|A_{0}+1+A_{0}^{-1}|\geq 1$. Denoting $a=A-A_0$ we get that  Eq.(\ref{e2}) is equivalent to
\begin{equation}
\begin{gathered}
\Big| 1+ \Big( (a+\!A_0) + (a+\!A_0)^{-1} +1 \Big)  \Big(-(a+A_0) - (a +A_0)^{-1} - 1 \Big)^{-s} \Big| \\
= \Big| 1- \Big( - (a +A_0)-(a +A_0)^{-1} -1\Big)^{-s} \Big|.
\end{gathered} \label{e3}
\end{equation}

Consider function
$$
G(a) = \Big(-(a+A_0) - (a+  A_0)^{-1} -1 \Big)^{-1}.
$$
Then Eq.(\ref{e3}) is equivalent to
$$
\Big| 1+ \Big( (a+\!A_0) + (a+\!A_0)^{-1} +1 \Big)  G(a)^{s} \Big| \\
= \Big| 1- G(a)^{s} \Big|
$$
and $|G(0)|=\frac{1}{|A_0+1+A_0^{-1}|} \leq 1$. Hence by Lemma \ref{lso}, for any positive real  number $\varepsilon > 0$ there exists $s_0$ such that for any $s\geq s_0$ the equation Eq.(\ref{e3}) has a solution in the disk $|a| < \varepsilon$. Thus, for $A_{0}$ there exists  $A'$ such that  $|A'-A_0|<\frac{\varepsilon}{2}$ and $A'$  satisfies the equation Eq.(\ref{e1}).

Since zeros of Eq.(\ref{e1}) are the limits of zeros of Yamada polynomials  $\{ R [C_{n}(\Theta_{s})] \}_{n=1,s=1}^{\infty, \infty}$,  there exist integers $n>0$ and $s>0$ such that $R[C_{n}(\Theta_{s})]$ has zero $A^{*}$ with $|A^{*}-A'|< \frac{\varepsilon}{2}$.

So there exists a zero $A^{*}$ of $R[C_{n}(\Theta_{s})]$ with $|A^{*}-A_0|\leq \varepsilon$, where complex number $A_0$ satisfies $|A_{0}+1+A_{0}^{-1}|\geq 1$.
\end{proof}

We remark that another proof of Theorem~\ref{tmain} follows from~\cite{SO} by the relations between graph polynomials. Indeed, Therem~7.2 of \cite{SO} states that if $|q_{0} -1| \geq 1$, then for each $\varepsilon >0$ there exists $s_{0} < \infty$ and $n_{0}(s) < \infty$ such that for all $s \geq s_{0}$ and $n \geq n_{0}$, the flow polynomial $F_{C_{n(\Theta_{s})}} (q)$ has a zero in the disc $| q - q_{0}| < \varepsilon$. Recall, see for example~\cite{AK}, that for planar graphs $G$ the Yamada polynomial coincides with a renormalization of the flow polynomial:
$$
F_{G}(q) = (-1)^{V(G) - E(G)} R_{G} (A),   \quad \text{where} \quad q = A +2 + A^{-1}
$$
with $|q_{0} - 1| \geq 1$ equivalent to $|A_{0} + 1 + A^{-1}_{0}| \geq 1$.

Next we consider zeros of Yamada polynomial of the second class of spatial graphs.

\begin{theorem}\label{tmain2}
Zeros of Yamada polynomials
$$
\{R[C_{n}(\Theta_{s}(\infty_{+}))]\}_{n=1, s=1}^{\infty, \infty}
$$
are dense in the region
$$
\{z\in \mathbb{C} \, : \,  |z^{-3}+2z^{-2}+z^{-1}+1|\leq |z^{-1}+1+z|\}.
$$
Zeros of the Yamada polynomials $\{ R[C_{n}(\Theta_{s}(\infty_{-}))] \}_{n=1, s=1}^{\infty, \infty}$ are dense in the region
$$
\{z\in \mathbb{C} \, : \,  |z^{3}+2z^{2}+z^1+1|\leq |z^{-1}+1+z|\}.
$$
\end{theorem}

\begin{proof}
By calculating the Yamada polynomial of $C_{n}(\Theta_{s}(\infty_{+}))$  we obtain
$$
R[C_{n}(\Theta_{s}(\infty_{+}))]= \Big( -R[\Theta_{s}(\infty_{+})] \Big)^n + \sigma \Big( \frac{R[\Theta_{s}(\infty_{+})]+R[\Theta_{s}(\infty_{+})']}{\sigma} \Big)^n.
$$
By Theorem \ref{tbj}, numbers $A$ satisfying the equation
\begin{equation}\label{e21}
|-R[\Theta_{s}(\infty_{+})]| = \Big| \frac{R[\Theta_{s}(\infty_{+})] + R[\Theta_{s}(\infty_{+})']}{\sigma} \Big|,
\end{equation}
are exactly limits of zeros of Yamada polynomials  $\{ R[C_{n} (\Theta_{s}(\infty_{+}))] \}_{n=1,s=1}^{\infty, \infty}$.

Observe that $\Theta_{s}'$ is the $s$-bouquet $B_{s}$. By Example \ref{E4} we have
$$
R[\Theta_{s}(\infty_{+})]=\frac{(-\sigma)^s+\sigma\big((\sigma+1)A^{-2}+1\big)^s}{1+\sigma}
$$
and
$$
R[\Theta_{s}(\infty_{+})']=(-1)^{s-1}\sigma^s.
$$
Therefore, Eq.(\ref{e21}) is equivalent to
$$
\begin{gathered}
\left| \frac{(-\sigma)^s+\sigma\big((\sigma+1)A^{-2}+1\big)^s}{1+\sigma} \right| \qquad \qquad \qquad \qquad \qquad \qquad  \qquad \qquad \\
=  \left| \frac{1}{\sigma} \left( \frac{(-\sigma)^s+\sigma\big((\sigma+1)A^{-2}+1\big)^s}{1+\sigma}+(-1)^{s-1}\sigma^s \right) \right| ,
\end{gathered}
$$
whence
\begin{equation}\label{e22}
\left|1+\sigma\Big( \frac{(\sigma+1)A^{-2}+1}{-\sigma} \Big)^{s}\right|=\left|1-\Big( \frac{(\sigma+1)A^{-2}+1}{-\sigma} \Big)^{s}\right|.
\end{equation}

Recall that $\sigma = A + 1 + A^{-1}$. Let $A_0$ be a complex number with
$$
|A_0^{-3}+2A_0^{-2}+A_0^{-1}+1|\leq |A_0^{-1}+1+A_0|.
$$
Denoting $a=A-A_0$ we get that  Eq.(\ref{e22}) is equivalent to
\begin{equation}
\begin{gathered}
\left| 1+ \Big( (a+\!A_0)+1+(a+\!A_0)^{-1}\Big) \right. \qquad \qquad \qquad \qquad  \\
\cdot \left. \left( \frac{\Big( (a+\!A_0)+2+(a+\!A_0)^{-1}\Big)(a+\!A_0)^{-2}+1}{-\Big( (a+\!A_0)+1+(a+\!A_0)^{-1}\Big)}\right)^s \right| \\
= \left| 1- \left( \frac{\Big( (a+\!A_0)+2+(a+\!A_0)^{-1}\Big)(a+\!A_0)^{-2}+1}{-\Big( (a+\!A_0)+1+(a+\!A_0)^{-1}\Big)}\right)^s \right|.
\end{gathered} \label{e23}
\end{equation}

Consider the function
$$
G(a) =  \frac{\Big( (a+\!A_0)+2+(a+\!A_0)^{-1}\Big)(a+\!A_0)^{-2}+1}{-\Big( (a+\!A_0)+1+(a+\!A_0)^{-1}\Big)} .
$$
Then Eq.(\ref{e23}) is equivalent to
$$
\Big| 1+ \Big( (a+\!A_0) + (a+\!A_0)^{-1} +1 \Big)  G(a)^{s} \Big| \\
= \Big| 1- G(a)^{s} \Big|
$$
and $|G(0)|=\Big| \frac{(A_0+2+A_0^{-1})A_0^{-2}+1}{A_0+1+A_0^{-1}} \Big| \leq 1$. Hence by Lemma \ref{lso}, for any positive real  number $\varepsilon > 0$ there exists $s_0$ such that for any $s\geq s_0$ the equation Eq.(\ref{e23}) has a solution in the disk $|a| < \varepsilon$. Thus, for $A_{0}$ there exists  $A'$ such that  $|A'-A_0|<\frac{\varepsilon}{2}$ and $A'$  satisfies the equation Eq.(\ref{e21}).

Since zeros of Eq.(\ref{e21}) are the limits of zeros of Yamada polynomials  $\{ R [C_{n}(\Theta_{s}(\infty_+))] \}_{n=1,s=1}^{\infty, \infty}$,  there exist integers $n>0$ and $s>0$ such that $R[C_{n}(\Theta_{s}(\infty_+))]$ has zero $A^{*}$ with $|A^{*}-A'|< \frac{\varepsilon}{2}$.

So there exists a zero $A^{*}$ of $R[C_{n}(\Theta_{s}(\infty_+))]$ with $|A^{*}-A_0|\leq \varepsilon$, where complex number $A_0$ satisfies $|A_0^{-3}+2A_0^{-2}+A_0^{-1}+1|\leq |A_0^{-1}+1+A_0|$.

By Proposition $6$ in \cite{YA}, the relation between the Yamada polynomials of a graph $g$ and its mirror image of a diagram $\widehat{g}$  is $R[g](A)=R[\widehat{g}](A^{-1})$.

Since $C_{n}(\Theta_{s}(\infty_-))$ is the mirror image of diagram $C_{n}(\Theta_{s}(\infty_+))$, the zeros of $R[C_{n}(\Theta_{s}(\infty_-))]$ are dense in the following region $\{z\in \mathbb{C} \, : \,  |z^{3}+2z^{2}+z^1+1|\leq |z^{-1}+1+z|\}$.

\end{proof}

Combining Theorem \ref{tmain} and Theorem \ref{tmain2}, we have

\begin{theorem}\label{tmain3}
Zeros of the Yamada polynomial of spatial graphs are dense in the following region:
$$
\Omega =
\big\{ z \in \mathbb C \, : \, |z+1+z^{-1}| \geq \min \{1, |z^{3}+2z^{2}+z+1|, |1 + z^{-1} + 2 z^{-2} + z^{-3} | \} \big\}.
$$
\end{theorem}

\section*{Acknowledgments}

The authors are thankful to Prof. Andrey Dobrynin for the helpful Maple's visualization of regions, discussed in the paper.

\bibliographystyle{amsplain}

\end{document}